\documentclass[a4paper,11pt]{amsart}

\newtheorem{thm}{Theorem}[section]
\newtheorem{cor}[thm]{Corollary}
\newtheorem{lem}[thm]{Lemma}

\newtheorem{prop}[thm]{Proposition}

\newtheorem{rem}[thm]{Remark}

\usepackage{amsmath,amsfonts,amssymb,amsthm,latexsym,bbm,color,hyperref}

\newcommand{\M}{\mathcal M}

\newcommand{\E}{\mathbb E}
\newcommand{\proba}{\mathbb P}

\newcommand{\Sph}{\mathbb S}

\newcommand{\one}{\mathbbm 1}

\newcommand{\R}{\mathbb R}
\newcommand{\Z}{\mathbb Z}

\newcommand{\C}{\mathbb{C}}
\newcommand{\Hq}{\mathbb{H}}

\newcommand{\iq}{\mathbf{i}}
\newcommand{\jq}{\mathbf{j}}
\newcommand{\kq}{\mathbf{k}}

\newcommand{\be}{\begin{equation}}
\newcommand{\ee}{\end{equation}}
\newcommand{\st}{such that }
\newcommand{\beq}{\begin{eqnarray*}}
\newcommand{\eeq}{\end{eqnarray*}}

\newcommand{\ste}{\, ;\, }

\newcommand{\Sp}{\mathbb{S}}

\newcommand{\f}{\frac}
\newcommand{\ff}{\frac{1}}

\newcommand{\pro}{probability }

\newcommand{\ie}{{\it i.e. }}

\def\restriction#1#2{\mathchoice
              {\setbox1\hbox{${\displaystyle #1}_{\scriptstyle #2}$}
              \restrictionaux{#1}{#2}}
              {\setbox1\hbox{${\textstyle #1}_{\scriptstyle #2}$}
              \restrictionaux{#1}{#2}}
              {\setbox1\hbox{${\scriptstyle #1}_{\scriptscriptstyle #2}$}
              \restrictionaux{#1}{#2}}
              {\setbox1\hbox{${\scriptscriptstyle #1}_{\scriptscriptstyle #2}$}
              \restrictionaux{#1}{#2}}}
\def\restrictionaux#1#2{{#1\,\smash{\vrule height .8\ht1 depth .85\dp1}}_{\,#2}}

\title[Gaussian quaternionic matrices]{Random right eigenvalues of Gaussian   quaternionic matrices}

\author[Florent Benaych-Georges]{Florent Benaych-Georges}\address{Florent Benaych-Georges, LPMA,  UPMC Univ Paris 6, Case courier 188, 4, Place Jussieu, 75252 Paris Cedex 05, France and CMAP, \'Ecole Polytechnique, route de Saclay, 91128 Palaiseau Cedex, France.}  \email{florent.benaych@gmail.com}

\author[Fran\c cois Chapon]{Fran\c cois Chapon} 
\address{Fran\c cois Chapon, T\'el\'ecom ParisTech, 46 rue Barrault, 75013 Paris, France.}
\email{francois.chapon@upmc.fr}

\date{\today}

\thanks{This work was partially supported by the \emph{Agence Nationale de la Recherche} grant ANR-08-BLAN-0311-03.}
\begin{document}

\maketitle

\begin{abstract}
We consider a random matrix whose entries are independent   Gaussian variables taking values in the field of quaternions with variance  $1/n$. Using logarithmic potential theory, we prove the almost sure convergence, as the dimension $n$ goes to infinity, of the empirical distribution of the right eigenvalues  towards some measure supported on the unit ball of the quaternions field. Some comments on more general Gaussian quaternionic random matrix models are also made.
\end{abstract}


\section*{Introduction}

Our motivation for studying quaternionic random matrices comes from the following facts. The projection onto the  complex plane of the uniform measure on the unit sphere $\Sph^3$ of $\R^4$ is the uniform measure on the unit disk $D(0,1)$ of $\C$, also called the circular law. Furthermore,  the projection onto the real axis of the uniform measure on $D(0,1)$ is the semi-circular law on $[-1,1]$. 
The last two measures play a key role in random matrix theory. Indeed, it is well known since Wigner's paper \cite{W58} that 
 as the dimension goes to infinity, the empirical distribution of the eigenvalues of a Gaussian Hermitian random matrix converges to the semi-circular law,   and it has also been proved that the empirical distribution of the   eigenvalues of a complex Gaussian random matrix converges to the circular law (see e.g. the book of Mehta \cite{mehta} or the paper of Tao, Vu and Krishnapur \cite{tvk10}).  Our initial idea, due to Philippe Biane, was   to
 find a random matrix model such that the empirical spectral measure would converge to the uniform measure on the unit sphere $\Sph^3$, and thus, in view of the previous observations, to study quaternionic random matrices, the unit sphere of quaternions being naturally  identified with the unit sphere $\mathbb S^3$. 
 The hope was then, after defining properly the eigenvalues of quaternionic random matrices, that the empirical spectral  distribution of a quaternionic Gaussian matrix will converge to the uniform distribution on the unit sphere of quaternions.   
 We will see that in fact, this is not the case, and we will prove a convergence result for the empirical spectral distribution towards some measure supported by the unit ball of the quaternions field.

The paper is organized as follows. In Section \ref{basicfactsquaternions}, we recall  some basic facts on the quaternions field $\Hq$ and on matrices of quaternions. 
Note that due to the noncommutativity of $\Hq$, two notions of eigenvalues cohabit,
 the left and the right ones, and we will only consider here the right eigenvalues.
We will see that we can associate to a quaternionic matrix $X(n)$ a complex matrix whose dimension is doubled which enables  to study   the complex right eigenvalues of $X(n)$.     
In Section \ref{convergence complex distribution}, we state the main results of this paper  about the convergence of the empirical distribution of right eigenvalues of quaternionic Gaussian random matrices.
The last sections and the appendix are devoted to the proofs of the results and some needed computations. 
\\

\noindent{\bf Acknowledgements. }Second author would like to thank Philippe Biane, L\'eonard Gallardo, Emmanuel Lesigne, Thierry L\'evy and Alain Rouault  for some useful comments.

\section{Basic facts on quaternionic matrices} \label{basicfactsquaternions}
\subsection{Quaternions}We begin with a brief recall of basic facts on quaternions (see \cite{z97}). Let us denote by $\Hq$ the noncommutative field of quaternions: as a real linear space, $\Hq$ admits a basis denoted by $(1, \iq,\jq,\kq)$ and its multiplicative structure is defined by the fact that $1$ is the neutral element, $\iq^2=\jq^2=\kq^2=-1$, $\iq\jq=-\jq\iq=\kq$, 
$\jq\kq=-\kq\jq=\iq$, $\kq\iq=-\iq\kq=\jq$. For all $q=q_0+q_1\iq+q_2\jq+q_3\kq\in \Hq$, one defines $q^*= q_0-q_1\iq-q_2\jq-q_3\kq$, $\Re(q)=q_0, \Im(q)=q_1\iq+q_2\jq+q_3\kq$. 
Then one has $qq^*=q_0^2+q_1^2+q_2^3+q_3^2$ (which proves that any non null quaternion is invertible) and we  put $|q|:=\sqrt{qq^*}$. It can be proved that for any $q,q'\in \Hq$, $|qq'|=|q||q'|$.
Note that identifying $1$ and $\iq$ with their usual definitions, one has $\R\subset \C\subset \Hq$ (an inclusion which is compatible with  the algebraic operations).

Two quaternions $x,y$ are said to be \emph{similar} if there exists a nonzero quaternion $q$ such that $x=qyq^{-1}$. Let $\Sph(\Hq)$ denote the group of quaternions with norm one.  Note that $x,y$ are said to be \emph{similar} if and only if for  a certain $u\in \Sph(\Hq)$, $x=uyu^*$.
The following lemma will be important for the study of right eigenvalues of quaternionic matrices.
\begin{lem}[Lemma 2.1 of  \cite{z97}]\label{lem_22811}
If $q=q_0+q_1\iq+q_2\jq+q_3\kq\in\Hq$, then $q$ and $\Re(q)+|\Im(q)|\iq$ are similar.
\end{lem}

\subsection{Right eigenvalues of quaternionic matrices}\label{sec:reqm}
Let $A\in M_n(\Hq)$. Then $\lambda\in\Hq$ is called a {\it right eigenvalue} of $A$ if there exists a non zero vector $X\in\Hq^n$ such that $AX=X\lambda$. If $\lambda$ is a right eigenvalue of $A$, we can easily see that $q\lambda q^{-1}$ is still a right eigenvalue for all $q\in\Hq\setminus\{0\}$. So the right spectrum of $A$ is either  infinite or contained in $\R$. Since every quaternion is similar to an unique element of $\C/(z\sim \bar z)$, we can first restrict our attention to complex right eigenvalues. Let  $A_1,A_2\in M_n(\C)$ be \st $A=A_1+A_2\jq $ (such matrices exist and are unique because $q_0+q_1\iq+q_2\jq+q_3\iq=(q_0+q_1\iq)+(q_2+q_3\iq)\jq$). Then    for any $\lambda\in\C$ and
$X=Y+Z\jq$, with $Y,Z\in\C^n$,  
 we have an equivalence between
\begin{enumerate}
\item[(i)]   $AX=X\lambda$,
\item[(ii)] $\begin{pmatrix}A_1 & A_2 \\ -\overline{A}_2 & \overline{A}_1 \end{pmatrix}
\begin{pmatrix}Y \\ -\overline Z \end{pmatrix}=\lambda \begin{pmatrix}Y \\ -\overline Z \end{pmatrix}$,
\item[(iii)] $\begin{pmatrix}A_1 & A_2 \\ -\overline{A}_2 & \overline{A}_1 \end{pmatrix}
\begin{pmatrix}Z \\ \overline Y \end{pmatrix}=\overline\lambda \begin{pmatrix}Z \\ \overline Y \end{pmatrix}$.
\end{enumerate}
Hence, the right spectrum of $A$, when restricted to complex numbers, is finite and is given by the $2n$ eigenvalues of the complex matrix $A':=\begin{pmatrix}A_1 & A_2 \\ -\overline{A}_2 & \overline{A}_1\end{pmatrix}$, which appear in conjugate pairs.
The whole set of right eigenvalues of $A$ is the union of the  similitary classes of the complex right eigenvalues of $A$. Note that when $A=X(n)$ is the Gaussian quaternionic matrix introduced below, $A'$ is said to be taken in the {\it Ginibre Symplectic Ensemble}.

\section{Main results}\label{convergence complex distribution}
\subsection{Quaternionic random matrix models}\label{def_model_14411}

Let $X(n)$ be an $n\times n$ quaternionic random matrix with independent identically distributed entries, whose distribution is defined by the fact that    $X(n)_{1,1}=q_0+q_1\iq+q_2\jq+q_3\kq$,
where $q_0,q_1,q_2,q_3$ are independent     Gaussian variables with mean $0$ and variance $1/(4n)$, so that \[
\E(X(n)_{i,j})=0 \quad \text{ and }\quad \E\big(|X(n)_{i,j}|^2\big)=\frac{1}{n}, \quad \text{for all } i,j=1,\ldots,n.
\]
We can associate to $X(n)$ an $2n\times 2n$ complex random matrix $Y(n)$ as seen above in Section \ref{sec:reqm}.
We will denote by $(z_{n,i})_{i=1}^{2n}$ the $2n$ eigenvalues of $Y(n)$, with the convention  
\[
z_{n,n+i}=\overline{z_{n,i}}, \quad \text{ for all } i=1,\ldots,n,
\]
because the eigenvalues of $Y(n)$ appear in conjugate pairs.
The right eigenvalues of $X(n)$ are then the elements of the similarity classes of the $z_{n,i}$'s, for $i=1,\ldots,n$.
The distribution of the eigenvalues of $Y(n)$ has been calculated by Ginibre \cite{gin} and admits the following density with respect to the Lebesgue measure on $\C^n$: for $z=(z_1,\dots,z_n)\in\C^n$,  
\[
P_n(z)=\frac{1}{c_n}\exp\Big(-2n\sum_{i=1}^n|z_i|^2\Big)
\prod_{1\leq i<j \leq n}|z_i-z_j|^2|z_i-\bar z_j|^2\prod_{i=1}^n|z_i-\bar z_i|^2,
\]
 where $c_n$ is a normalization constant.

If we write
\[
\sum_{1\leq i\not = j\leq 2n } \log |z_i-z_j|=2 \sum_{1\leq i\not = j\leq n}\log |z_i-z_j||z_i-\bar z_j| 
+2\sum_{i=1,\ldots,n}\log |z_i-\bar z_i|,
\]
then the density $P_n$ can be rewritten as
\[
P_n(z)=\frac{1}{c_n}\exp\left(-\frac{1}{2} \bigg(2n\sum_{i=1}^{2n}|z_i|^2+\sum_{1\leq i\not = j\leq 2n } \log |z_i-z_j|^{-1}+\sum_{i=1}^{2n}\log |z_i-\bar z_i|^{-1} \bigg)\right).
\]

We will embed our considerations in a more general setting, since
the density $P_n$ can be generalized in the following way. Let $V:\C\rightarrow[0,+\infty]$ be a continuous non-negative function, conjugate invariant, and such that there is a constant $\delta>0$ such that $V(z)\geq (\delta+1)\log(|z|^2+1)$, for $|z|$ large enough. We suppose that the set $E_0=\{z\in\C\,|\, V(z)<\infty\}$ has positive capacity, \ie that there exists at least one   \pro measure $\mu$ with compact support contained in $E_0$    such that 
$ \int\!\!\int \log |x-y|^{-1} d\mu(x)d\mu(y) <\infty$ (this notion will become apparent in Section \ref{proofthmthcv14411}).
  For $z=(z_1, \ldots, z_n)\in\C^n$, set 
\[
K_n(z):=\sum_{1\leq i\not = j\leq 2n}k(z_i,z_j) \quad\textrm{ with }\quad k(x,y):=\log|x-y|^{-1}+\frac{1}{2}(V(x)+V(y)).
\]
Then we have   
\begin{align*}
K_n(z)
&=\sum_{1\leq i\not = j\leq 2n} \log|z_i-z_j|^{-1}+\sum_{1\leq i\not = j\leq 2n} V(z_i) \\
&=\sum_{1\leq i\not = j\leq 2n} \log|z_i-z_j|^{-1}+(2n-1)\sum_{i=1,\ldots,2n}V(z_i).
\end{align*}
Define the probability density $P_n^V$ on $\C^n$ by,
\[
P_n^V(z)=\frac{1}{c_n^V}\exp\left(-\frac{1}{2}\bigg( K_n(z) +\sum_{i=1}^{2n}V(z_i)+\sum_{i=1}^{2n}\log |z_i-\bar z_i|^{-1}\bigg) \right),
\]
where $c_n^V$ is the normalization constant.
One  recovers $P_n$ for  $V(x)=|x|^2$. Note also that the matrices having the structure 
$\begin{pmatrix}A_1 & A_2 \\ -\overline{A}_2 & \overline{A}_1 \end{pmatrix}$ are   non normal, so  when $V(z)$ is not of the form  $V(z)=a_1z^2+a_2\bar z^2+a_3z\bar z$, the density $P_n^V$ is     not the density of the complex eigenvalues of a quaternionic random matrix with distribution $\propto \exp(\operatorname{Trace} (b_1 A^2+b_2{A^*}^2+b_3 AA^*))d A$.

\subsection{Convergence of the empirical measure associated to a $P_n^V$-distributed sample}
The following theorem is proved for any potential $V$ as above. 
\begin{thm}\label{thcv14411}
Let $z^{(n)}=(z_{n,1},\ldots,z_{n,n})\in\C^n$ be distributed according to $P^V_n$. Then, the empirical distribution
\[
 \frac{1}{2n}\sum_{i=1}\delta_{z_{n,i}}
\]
converges almost surely, as $n$ goes to infinity, towards a compactly supported \pro  measure $\mu^V$, which is the unique minimum, on the set of  \pro measures on $\C$, of the  \emph{weighted logarithmic  energy}, defined    by the formula 
\[
I(\mu):=\int\int \log |x-y|^{-1} d\mu(x) d\mu(y) +\int V(z) d\mu(z),
\]
for $\mu$ a \pro measure on $\C$.
\end{thm}

Taking $V(z)=|z|^2$ in the last theorem and identifying $\mu^V$, we get the following corollary about the convergence of the empirical distribution of the complex right eigenvalues of the Gaussian random quaternionic matrix $X(n)$ introduced in Section \ref{def_model_14411}.

\begin{cor} \label{VP.complexes.25.11.08}
Let  $z_{n,1}, \ldots, z_{n,2n}$ be the complex right  eigenvalues of $X(n)$.  As $n$ tends to infinity, the   \pro measure 
\[
\ff{2n}\sum_{i=1}^{2n}\delta_{z_{n,i}}
\]
 tends  almost surely to the uniform law $\mu$ on the unit disk of the complex plane.
\end{cor}

\subsection{Limits of uniformly chosen right eigenvalues of quaternionic random matrices}
 
\begin{thm} \label{VP.quat.25.11.08}
Let  $X(n)$ be as defined in Section \ref{def_model_14411}.    Let $C_{n,1},\ldots, C_{n,n}$ be the   similarity classes of its right spectrum, and $c_{n,1}, \ldots, c_{n,n}$ be elements taken independently at random, uniformly in respectively $C_{n,1}, \ldots, C_{n,n}$. Then as $n\to\infty$, the empirical \pro measure $$\ff{n}\sum_{i=1}^{n}\delta_{c_{n,i}}$$ tends  almost surely to the law on $\Hq$ with density with respect to the Lebesgue measure  \be\label{26.11.08.2}\rho(q)=\ff{2\pi^2 |\Im q|^2}\one_{ |q|\leq 1}.\ee 
\end{thm}

\begin{rem}[Alternative formulations of  Theorem \ref{VP.quat.25.11.08}] By Lemma \ref{lem_22811}, we have ${C_{n,i}}=\{q\in \Hq, \,\;\, \Re (q)=\Re (c_{n,i}),\; |\Im (q)|=|\Im (c_{n,i})|\}$, hence ${C_{n,i}}$  is a sphere of an affine three-dimensional subspace, with radius $r_{n,i}:=|\Im (c_{n,i})|$. Let us  denote by $\restriction{\mathcal S}{C_{n,i}}$ the uniform measure on ${C_{n,i}}$, with a normalization inherited from the euclidian structure of $\Hq$, i.e. with total mass $4\pi r_{n,i}^2$ and set $V_n:=\sum_{i=1}^n4\pi r_{n,i}^2$ to be the total mass of the right spectrum of $X(n)$. 
It can easily be deduced from the proof of Theorem \ref{VP.quat.25.11.08} that this theorem  is equivalent to its following alternative version:
as $n$ goes to infinity, both random \pro measures \be\label{mes_alt}\ff{V_n}\sum_{i=1}^{n}4\pi r_{n,i}^2\delta_{c_{n,i}}\quad\textrm{ and }\quad \ff{V_n}\sum_{i=1}^{n}\restriction{\mathcal S}{C_{n,i}}\ee converge weakly to the uniform \pro measure on the unit ball of $\Hq$. Note that  both \pro measures of  \eqref{mes_alt} correspond to the picking of random right eigenvalues of $X(n)$ at random, with probabilities proportional  to the masses of their similarity classes. The second \pro measure of  \eqref{mes_alt} is, in a certain sense, the exact empirical distribution of the right eigenvalues of $X(n)$. As a consequence, an alternative formulation of Theorem \ref{VP.quat.25.11.08} is the following one: {\bf the empirical distribution on the right spectrum of $X(n)$ converges to the uniform \pro measure on the unit euclidian ball of $\Hq$}.
\end{rem}

We see from Theorem \ref{VP.quat.25.11.08} that the empirical spectral distribution of $X(n)$ does not tend to  the uniform measure on the unit   sphere $\Sp(\Hq)$ of $\Hq$. One can ask if it is possible to find a quaternionic random matrix model for which the empirical distribution of right eigenvalues would converge to the uniform measure on $\mathbb S(\Hq)$. We next show that it is not possible if one restricts to quadratic potentials, \ie potentials of the form $V(z)=a_1z^2+a_2\bar z^2+a_3z\bar z$ corresponding to a more general Gaussian quaternionic random matrix model. 

 \begin{thm} \label{nonquadpotthm}
There is no  quadratic potential $V$ such that for a $P_n^V$-distributed random vector 
$(z_{n,1},\ldots,z_{n,2n})$ and some independent  $u_i$'s uniformly distributed on $\Sph(\Hq)$, the random measure  $\frac{1}{2n}\sum_{i=1}^{2n}\delta_{u_iz_{n,i}u_i^*}$ converges to the uniform measure of $\Sph(\Hq)$.
\end{thm} 

Note that it remains open whether there are some more general quaternionic random matrix models whose spectral empirical distribution  converges towards the uniform distribution.

\vskip 1.5cm

\section{Proof of Theorem \ref{thcv14411}} \label{proofthmthcv14411}

\vskip 0.5cm

The proof is quite classical. The arguments are close to the ones of \cite{bag97,baz98,deift,johansson-fluctuations,alice-greg-ofer} and use logarithmic potential theory.\\

Firstly, the fact that the functional $I$ admits a minimum  achieved at a unique \pro measure $\mu^V$, which is compactly supported, is the statement of Theorem~1.3 p. 27  of  \cite{saff}. This measure is called the \emph{equilibrium measure} of $V$.
Let us denote by $E^V$ the minimum of $I$, \ie $E^V=I(\mu^V)$.

For $x\in\C$, let $H(x)=V(x)-\log(|x|^2+1)$. We remark that $k(x,x')\geq \frac{H(x)+H(x')}{2}$, for all $x,x'\in\C$, since $|x-x'|^2\leq (|x|^2+1)(|x'|^2+1)$. So, $k$ is bounded from below. 
Let $l\geq 0$. For $\mu$ a \pro measure on $\C$, let us define
\[
k^l=k\wedge l, \; I^l(\mu)=\int k^l(x,y) d\mu(x) d\mu(y) \text{ and } K^l_n(z)=\sum_{1\leq i\not = j\leq 2n} k^l(z_i,z_j).
\]
Since $k^l$ is continuous almost everywhere, by monotone convergence, $I^l$ converges to $I$ as $l$ goes to infinity.

\noindent{\bf Notation.} {\it
For $z=(z_1,\ldots,z_n)\in\C^n$, we will denote by $\mu_z$ the empirical distribution of $z_1,\ldots,z_{2n}$ (with our convention that $z_{n+i}=\bar z_i$, for $i=1,\ldots,n$), \ie
\[
\mu_z:=\frac{1}{2n}\sum_{i=1}^{2n}\delta_{z_i}=\frac{1}{2n}\sum_{i=1}^n(\delta_{z_i}+\delta_{\bar z_i}).
\]
}

An immediate calculation gives 
\be\label{majIl}I^l(\mu_z)=\frac{1}{4n^2}K^l_n(z)+\frac{l}{2n}.\ee
The following two facts will be useful for the proof of the convergence of the empirical distribution.
\begin{lem} \label{tightlem}
\begin{enumerate}
\item[(i)] For all $n$, consider    $z^{(n)}\in\C^n$ such that the sequence $\frac{1}{4n^2}K_n(z^{(n)})$ is bounded. Then the sequence   $\mu_{z^{(n)}}$ is tight.
\item[(ii)] Let $\mu_n$ be a tight sequence of probability measures on $\C$ such that for all $l$, 
\[
\limsup_n I^l(\mu_n)\leq E^V.
\] 
Then $\mu_n$ converge  weakly to $\mu^V$.
\end{enumerate}
\end{lem}
\begin{proof}
 (i) 
follows from the facts that $H(x)\to\infty$ as $|x|\to\infty$ and $k(x,y)\geq \frac{H(x)+H(y)}{2}$.

  (ii) Let $\sigma$ be an accumulation point of   $\mu_n$. The functional  $I^l$ is continuous for the weak convergence because $k^l$ is bounded and continuous a.e., so 
$
I^l(\sigma)\leq E^V
$
by hypothesis. Letting $l\to\infty$, we get $I(\sigma)\leq E^V$. By unicity of the equilibrium measure, the result follows. 
\end{proof}

From this lemma and Equation  \eqref{majIl}, we deduce the following proposition.
\begin{prop} \label{convded}
Let, for all $n$, $x^{(n)}\in\C^n$ such that 
\[
\limsup_n \frac{1}{4n^2}K_n(x^{(n)})\leq E^V.
\]
Then $\mu_{x^{(n)}}$ converges weakly  to $\mu^V$ as $n$ goes to infinity.
\end{prop}

We shall now prove a concentration result of $\frac{1}{4n^2}K_n(z)$, under the distribution $P_n^V$,  which is the analogue of the same result in the Hermitian case.
Firstly, the following lemma gives some estimate on the normalization constant $c_n^V$. The proof follows exactly the lines of the analogue result for the Hermitian case, so we refer to \cite{deift} or \cite{johansson-fluctuations} for the details.

\begin{lem}\label{renorm}
As $n$ goes to infinity, 
$
\limsup \frac{1}{2n^2}\log \frac{1}{c_n^V} \leq E^V.
$
\end{lem}


The concentration of $\frac{1}{4n^2}K_n(z^{(n)})$ around $E^V$ is the content  of the next proposition.
\begin{prop} \label{concentr}
Let $z^{(n)}$ be a random vector of $\C^n$ having  density $P_n^V$.
Then, for all $0<\varepsilon<\eta$,  for $n$ large enough,
\[
\proba\left(\ff{4n^2}K_n(z^{(n)})>E^V+\eta\right)\leq e^{-2\varepsilon n^2}.
\]
\end{prop}
\begin{proof}Set $
A_n:= \{x\in\C^n \,\big|\, \frac{1}{4n^2}K_n(x)> E^V+\eta\}.$
We have,
\begin{align*}
\proba\left(z^{(n)}\in A_n\right)
&\leq \frac{1}{c_n^V}e^{-2n^2(E^V+\eta)}\int e^{-\frac{1}{2}\big(\sum_{i=1}^{2n}V(z_i)+\sum_{i=1}^{2n}\log |z_i-\bar z_i|^{-1}\big)} d z_1\cdots d z_n \\
&=\frac{1}{c_n^V}e^{-2n^2(E^V+\eta)} \left(\int e^{-V(x)-\log |x-\bar x|^{-1}}  d x \right)^{n}
\end{align*}
Now, taking the logarithm  of the above inequality, we obtain
\[
\frac{1}{2n^2} \log \proba\left(z^{(n)}\in A_n \right) \leq 
\frac{1}{2n^2} \log \frac{1}{c_n^V}-E^V-\eta +\frac{1}{2n}\log (a),
\]
where $a=\int_{\C}e^{-V(x)}|x-\bar x|dx$ is a finite constant by the hypothesis $V(x)\geq (1+\delta)\log(|x|^2+1)$ for $|x|$ large enough. Now, taking the $\limsup$ and using the estimate of $c_n^V$ given by Lemma \ref{renorm}, we obtain
\[
\limsup_n \frac{1}{2n^2}\log \proba\left(z^{(n)}\in A_n \right)\leq E^V-E^V-\eta <-\varepsilon,
\]
so the result follows.
\end{proof}

\begin{proof}[Proof of Theorem  \ref{thcv14411}]

By Proposition \ref{concentr} and Borel-Cantelli Lemma, we have that almost surely, for all $\eta>0$, and $n$ large enough, $\frac{1}{4n^2}K_n(z^{(n)})\leq E^V+\eta$. Hence, almost surely 
$\limsup_n \frac{1}{4n^2}K_n(z^{(n)})\leq E^V$. By Proposition \ref{convded}, the almost surely convergence of $\mu_{z^{(n)}}$ follows.\end{proof}

\vskip 1.5cm

\section{Proof of Corollary \ref{VP.complexes.25.11.08}:
 identification of $\mu^V$ for $V(z)=|z|^2$} 
 
 \vskip 0.5cm
 
  For $\sigma$ a \pro measure on $\C$ and $x\in\C$, let 
\[
U^\sigma(x):=\int\log|x-y|^{-1} d\sigma(y),
\]
be the  \emph{logarithmic potential} of some measure $\sigma\in\M(\C)$. The equilibrium measure $\mu^V$ can be characterized by the following theorem, see Remark~1.5 p. 28 and Theorem~3.3 p. 44 of \cite{saff}. 
\begin{thm} \label{caracpot}
Let $\sigma\in\M(\C)$, be a compactly supported \pro measure on $\C$ \st $I(\sigma)<\infty$.   Then we have $\sigma=\mu^V$ if and only if there exists a constant $l$ such that
\begin{enumerate}
\item 
$2U^\sigma(x)+V(x)=l$, $\mu^V$-a.e. on the support of $\sigma$
\item
$2U^\sigma(x)+V(x)\geq l$, $\mu^V$-a.e. 
\end{enumerate}
\end{thm}

Let us now apply this theorem to find out 
the equilibrium measure of the quadratic potential $V(z)=|z|^2$ is  the uniform measure $\mu$ on $D(0,1)$. We have to calculate the logarithmic potential of $\mu$.
 
 Firstly, the mean-value property of harmonic functions implies the following useful classical lemma.
\begin{lem} \label{meanclassical}
For all $x\in\C$,
\[
\int_{-\pi}^{\pi} \log|x-re^{i\theta}| d\theta 
=
\begin{cases}
2\pi \log r,  & \text{if $|x|\leq r$,} \\
2\pi \log |x|, & \text{if $|x|>r$.}
\end{cases}
\]
\end{lem}

A straightforward computation using the above lemma allows one to compute the logarithmic potential of $\mu$ given in the next proposition.
\begin{prop}
Let $\mu$ be the uniform measure on the unit disk $D(0,1)$. Let $U^\mu$ be the logarithmic potential of $\mu$, \ie $U^\mu(x)=\int \log |x-y|^{-1}\mu(dy)$ for $x\in\C$. Then,
\[
U^\mu(x)=
\begin{cases}
\frac{1}{2}(1-|x|^2),& \text{for $|x|\leq 1$,}\\
-\log |x|,& \text{for $|x|>1$.}
\end{cases}
\]
\end{prop}

Now, as a direct consequence of Theorem \ref{caracpot},   the equilibrium measure of the weighted logarithmic energy $I$ with potential $V(z)=|z|^2$ is given by the uniform measure $\mu$ on the unit disk of $\C$.

\vskip 1.5cm

\section{Proof of Theorem \ref{VP.quat.25.11.08}} 
 
 \vskip 0.5cm
 
 
 Before   providing a formal proof of Theorem \ref{VP.quat.25.11.08}, let us outline  this proof. We will at first consider the random \pro  measure \be\label{mes_20041118h21}\frac{1}{n}\sum_{i=1}^n \delta_{(z_{n,i},u_i)},\ee where the $z_{n,i}$'s are the complex right eigenvalues of $X(n)$ with positive imaginary part and the $u_i$'s are i.i.d.   variables, independent of the $z_{n,i}$'s, and  whose distribution $\restriction{\mathcal U}{\Sp(\Hq)}$ is the uniform law   on the sphere $\Sp(\Hq)$. We know, by Corollary \ref{VP.complexes.25.11.08}, that $\frac{1}{n}\sum_{i=1}^n \delta_{z_{n,i}}$ converges  towards  $\restriction{\mathcal U}{D^+_\C(0,1)}$, the uniform distribution on the intersection of the unit circle with the upper half-plane. We also know, by the Law of Large Numbers, that $\frac{1}{n}\sum_{i=1}^n \delta_{u_{i}}$ converges towards   $\restriction{\mathcal U}{\Sp(\Hq)}$. Hence by independence and exchangeability of the samples $(z_{n,1}, \ldots, z_{n,n})$ and $(u_1, \ldots, u_n)$, the measure of \eqref{mes_20041118h21} will converge to $\restriction{\mathcal U}{D^+_\C(0,1)}\otimes\restriction{\mathcal U}{\Sp(\Hq)}$ (this is the content of Lemma \ref{empirique+tensor.25.11.08} below, whose statement seems quite intuitive and well-known, but for which we did not find any reference, hence provide a proof).  To conclude the proof, we next have to prove that the push-forward of $\restriction{\mathcal U}{D^+_\C(0,1)}\otimes\restriction{\mathcal U}{\Sp(\Hq)}$ by the map $(z,u)\mapsto uzu^*$ is the   measure of \eqref{26.11.08.2}. This last step  relies on the identification of the orbits of the action of ${\Sp(\Hq)}$ by conjugation on $\Hq$, which is due to   Brenner  in \cite{Brenner51}. A kind of probabilistic version of Brenner's result is given in Proposition   \ref{26.11.08.1} below. 
\\

So first, we state the    preliminary results that we need, the first of which is a kind of multidimensional and almost sure reformulation of \cite[Th. 1.8, Part 1]{L-BEN}.

\begin{lem}\label{empirique+tensor.25.11.08}Let $p,q$ be two positive integers and $\mu,\nu$ be two  \pro measures on respectively $\R^p,\R^q$.  Let, for each $n$, $$x_{n,1}, \ldots, x_{n,n},y_{n,1}, \ldots, y_{n,n}$$be random variables, the $x_{n,i}$'s taking values in $\R^p$ and the $y_{n,j}$'s taking values in $\R^q$ \st 
\begin{itemize}\item[(i)] for each $n$,  $\{x_{n,1}, \ldots, x_{n,n}\}$ and $\{y_{n,1}, \ldots, y_{n,n}\}$ are independent sets of random variables,
\item[(ii)]  for each $n$,  the laws of the random vectors $(x_{n,1}, \ldots, x_{n,n})$ and $(y_{n,1}, \ldots, y_{n,n})$ are invariant under the actions of the symetric group $\mathcal S_n$,
\item[(iii)] the random \pro measures $$\mu_n:=\ff{n}\sum_{i=1}^n\delta_{x_{n,i}},\quad \nu_n:=\ff{n}\sum_{i=1}^n\delta_{y_{n,i}}$$ converge almost surely respectively to $\mu,\nu$ 	as $n$ tends to infinity.
\end{itemize}
Then the random \pro measure $$\rho_n:=\ff{n}\sum_{i=1}^n\delta_{(x_{n,i},y_{n,i})}$$converges almost surely to $\mu\otimes\nu$ as $n$ tends to infinity.
\end{lem}

\begin{proof}
It suffices to prove that for any compactly supported real valued continuous function $f$ defined on $\R^{p+q}$, $\int f(t) d \rho_n(t)$ tends almost surely to $\int f(t) d \mu\otimes \nu(t)$ as $n$ tends to infinity. By Stone-Weierstrass theorem, it suffices to prove it when $f=g\otimes h$, with $g,h$  compactly supported real valued continuous functions defined respectively on $\R^{p}, \R^{q}$. Moreover, since by (iii), it is obvious when $g$ or $h$ is a constant function, one can suppose that $$\int g(t) d \mu(t)=\int h(t) d \nu(t)=0.$$ Thus we have to prove that $\ff{n}\sum_{i=1}^ng(x_{n,i})h(y_{n,i})$ tends almost surely to $0$. Let us define, for all $n$, for all $i=1, \ldots, n$, $$a_{n,i}=g(x_{n,i})-\ff{n}\sum_{j=1}^n g(x_{n,j}),\quad b_{n,i}=h(y_{n,i})-\ff{n}\sum_{j=1}^n h(y_{n,j}).$$ Since $\ff{n}\sum_{j=1}^n g(x_{n,j})$ and $\ff{n}\sum_{j=1}^n h(y_{n,j})$ converge almost surely to zero and the functions $g,h$ are bounded, it suffices to prove that $\ff{n}\sum_{i=1
 }^na_{n,i}b_{n,i}$ converges almost surely to zero. Note that the advantage of working with the $a_{n,i}$'s and the $b_{n,i}$'s instead of working with the $ g(x_{n,i})$'s and the $ h(y_{n,i})$'s is that for all $n$, one has almost surely \be\label{sum.null.25.11.08} \sum_{i=1}^n a_{n,i}=\sum_{i=1}^n b_{n,i}=0.\ee

We claim that the fourth moment of $\sum_{i=1}^na_{n,i}b_{n,i}$ is $\operatorname{O}(n^2)$. Let us prove it. We have  \be\label{1658.25.11.08}\E\left[\left(\sum_{i=1}^na_{n,i}b_{n,i}\right)^4\right]=\sum_{1\leq i,j,k,l\leq n}\E(a_{n,i}a_{n,j}a_{n,k}a_{n,l})\E(b_{n,i}b_{n,j}b_{n,k}b_{n,l}).\ee Note that by  the hypothesis (ii), each term in the sum of the  right hand side of \eqref{1658.25.11.08} only depends on the partition of $\{1,2,3,4\}$ defined by $(i,j,k,l)$: let us define \beq \alpha_{n,1,1,1,1}&=& \E(a_{n,i}a_{n,j}a_{n,k}a_{n,l})\quad \textrm{ for $i,j,k,l$ pairwise distinct,}\\ 
\alpha_{n,2,1,1}&=& \E(a_{n,i}^2a_{n,k}a_{n,l})\quad \textrm{ for $i,k,l$ pairwise distinct,}\\ 
\alpha_{n,2,2}&=& \E(a_{n,i}^2 a_{n,k}^2)\quad \textrm{ for $i\neq k$,}\\ 
\alpha_{n,3,1}&=& \E(a_{n,i}^3a_{n,l})\quad \textrm{ for $i\neq l$,}\\ 
\alpha_{n,4}&=& \E(a_{n,i}^4)\quad \textrm{ for $i\in\{1, \ldots, n\}$}\eeq and let us define the $\beta_{n,\cdot}$'s in the same way with the $b_{n,i}$'s  instead of the $a_{n,i}$'s.
Then by  the hypothesis (ii), with the notation $A_n^k=n(n-1)\cdots (n-k+1)$ for all $k$, 
\begin{multline*}
\E\left[\left(\sum_{i=1}^na_{n,i}b_{n,i}\right)^4\right]= A_n^4 \alpha_{n,1,1,1,1}\beta_{n,1,1,1,1}+6A_n^3\alpha_{n,2,1,1}\beta_{n,2,1,1}\\
+4A_n^2(\alpha_{n,2,2}\beta_{n,2,2}+\alpha_{n,3,1}\beta_{n,3,1})+n \alpha_{n,4} \beta_{n,4}.
\end{multline*}
Note that since $g$ and $h$ are bounded, the $\alpha_{n,\cdot}$'s and the $\beta_{n,\cdot}$'s are all $\operatorname{O}(1)$. Thus to prove the claim, it suffices to prove that $\alpha_{n,1,1,1,1}=\operatorname{O}(n^{-1})$,  $\alpha_{n,2,1,1}=\operatorname{O}(n^{-1})$ and that the same holds for the $\beta$'s. We shall only treat the case of the $\alpha$'s. Passing \eqref{sum.null.25.11.08}  to the fourth power and taking the expectation, we get $$ A_n^4 \alpha_{n,1,1,1,1}+6A_n^3\alpha_{n,2,1,1}+4A_n^2(\alpha_{n,2,2}+\alpha_{n,3,1})+n \alpha_{n,4} =0,$$from which it follows that $\alpha_{n,1,1,1,1}=\operatorname{O}(n^{-1})$. In the same way, it follows from \eqref{sum.null.25.11.08} that 
$$\sum_{\substack{i,j,k\\ \textrm{pairwise distinct}}}a_{n,i}^2a_{n,j}a_{n,k}=\sum_{i\neq j}a_{n,i}^2a_{n,j}(-a_{n,i}-a_{n,j}),$$which gets, by integration, $A_n^3 \alpha_{n,2,1,1}=-A_n^2(\alpha_{n,3,1}+\alpha_{n,2,2}).$ Hence $\alpha_{n,2,1,1}=\operatorname{O}(n^{-1}).$ 

Thus we have proved that the fourth moment of $\sum_{i=1}^na_{n,i}b_{n,i}$ is $\operatorname{O}(n^2)$. 
By Markov's inequality and Borel-Cantelli's lemma, $\ff{n}\sum_{i=1}^na_{n,i}b_{n,i}$ tends almost surely to zero. 
\end{proof}

\begin{prop}\label{26.11.08.1}Fix $z_0\in \Hq$ and let $u$ be a Haar distributed random element of $\Sph(\Hq)$.   Then $uz_0u^*$ has the law of $\Re(z_0)+|\Im(z_0)|S$, for $S$ uniformly distributed on the unit sphere of   $\Hq_0:=\{q\in \Hq\ste \Re(q)=0\}$. 
\end{prop}

To prove the proposition, we shall use the following lemma. 

\begin{lem}\label{unf.25.11.08}
Let $E$ be a metric space, $\mu$ a probability measure on $E$, and $G$ a compact group such that $G$ acts transitively on $E$ and preserves $\mu$. Let $x\in E$. Then the push-forward of the Haar measure on $G$ by the map $g\mapsto g\cdot x$ coincides with $\mu$.
\end{lem}

\begin{proof} Let $f$ be a bounded measurable function on $E$ and let $\mu_G$ denote the  Haar measure on $G$. For all $y\in E$ there exists $g_y\in G$ such that $x=g_y\cdot y$. As a consequence, by right-invariance of $\mu_G$, $\int_G f(g\cdot y)d\mu_G(g)   $ does not depend on $y$. So 
 \beq \int_G f(g\cdot x) d\mu_G (g)&=&\int_{E}  \int_G f(g \cdot  y)d\mu_G(g) d\mu (y) \\ & =&\int_G  \int_{E}f(g\cdot y)d\mu (y) d\mu_G(g) \\
  & =&\int_G  \int_{E}f(  y)d\mu (y) d\mu_G(g) \\ & =&\int_{E} f(y)d\mu (y),
\eeq
which proves that $g\cdot x$ is distributed according to $\mu$.
\end{proof}

\begin{proof}[Proof of Proposition \ref{26.11.08.1}]Let us endow $\Hq$ with its canonical Euclidian structure (for which $(1, \iq,\jq,\kq)$ is an orthogonal basis). Then the action of the group of quaternions with norm one on $\Hq$ defined by $u\!\cdot\! z:=uzu^*$ is linear, norm preserving (since the norm is multiplicative on $\Hq$) stabilizes $\R$ and thus also the orthogonal of $\R$ in $\Hq$, namely the space of quaternions with null real part. Moreover, by Brenner's Theorem (see e.g. \cite[Th. 2.2]{z97}), the action induced on the unit sphere of this subspace is transitive, thus the proposition follows from the previous lemma.
\end{proof}

\begin{proof}[Proof of Theorem \ref {VP.quat.25.11.08}.]
Let, for each $n$,  $z_{n,1}, \ldots, z_{n,2n}$ be the complex right eigenvalues of $X(n)$, ordered in such a way that for all $i=1, \ldots, n$, $z_{n,n+i}=\overline{z_{n,i}}$ and the imaginary part of $z_{n,i}$ is positive and that the joint law of $(z_{n,1}, \ldots, z_{n,n})$ is invariant under the action of the symmetric group $\mathcal S_n$.  Then the conjugation classes of its right spectrum are 
\[
C_{n,1}:=\{uz_{n,1}u^*\ste u\in \Hq, |u|=1\},\ldots, C_{n,n}:=\{uz_{n,n}u^*\ste u\in \Hq, |u|=1\},
\] 
and $c_{n,1}, \ldots, c_{n,n}$ can be defined by $c_{n,1}=u_1z_{n,1}u_1^*, \ldots, c_{n,n}=u_nz_{n,n}u_n^*$ for $(u_i)_{i\geq 1}$ a family of independent random variables with uniform distribution on the group of the quaternions with norm one, \st $\{z_{n,1}, \ldots, z_{n,n}\}$ and $\{u_i\ste i\geq 1\}$ are independent. 

Note that the random \pro measure $\ff{n}\sum_{i=1}^n \delta_{z_{n,i}}$ is the push-forward of   $\ff{2n}\sum_{i=1}^{2n}\delta_{z_{n,i}}$ by the map $z\in \C\mapsto \Re(z)+\iq |\Im(z)|.$ Thus by Corollary \ref{VP.complexes.25.11.08}, it converges almost surely, as $n$ tends to infinity, to the push-forward of the uniform law on the unit circle of the complex plane by this map, i.e. to the uniform law on the intersection of the unit circle with the upper half-plane, denoted by $\restriction{\mathcal U}{D^+_\C(0,1)}$. Moreover, by the strong law of large numbers, $$\ff{n}\sum_{i=1}^n \delta_{u_i}$$ converges almost surely, as $n$ tends to infinity, to the uniform law on  the group of the quaternions with norm one, denoted by $\restriction{\mathcal U}{\Sp(\Hq)}$.

So by Lemma \ref{empirique+tensor.25.11.08},  as $n$ tends to infinity,  the random \pro measure $$\ff{n}\sum_{i=1}^n \delta_{(z_{n,i},u_i)}$$
converges almost surely to $\restriction{\mathcal U}{D^+_\C(0,1)}\otimes \restriction{\mathcal U}{\Sp(\Hq)}$. Thus the random \pro measure $$\ff{n}\sum_{i=1}^n \delta_{c_{n,i}}$$ converges almost surely, as $n$ tends to infinity, to the push-forward of the law $\restriction{\mathcal U}{D^+_\C(0,1)}\otimes \restriction{\mathcal U}{\Sp(\Hq)}$ by the map $(z,u)\mapsto uzu^*$.

To prove that this law, that we shall denote by $L$, has a density given by \eqref{26.11.08.2}, it suffices to note that by Proposition \ref{26.11.08.1}, for any Borel bounded function $f$ on $\Hq$, 
\[ 
\int f(t) d L(t)= \f{2}{\pi}\int_{x= -\infty}^{+\infty}\int_{y=0}^{+\infty}\int_{s\in \mathbb{S}(\Hq_0)}1_{x^2+y^2\leq 1}f(x+ys) d s d y d x,
\]
where $\mathbb{S}(\Hq_0)$  denotes the unit sphere of the subspace $\Hq_0$ of $\Hq$ of quaternions with null real part and $ d s$ denotes the uniform \pro measure on this sphere. By spherical integration on the three-dimensional space of quaternion with null real part, we have   
\begin{eqnarray*} 
\int f(t) d L(t)\!\!\!\! 
&=& \!\!\!\!\f{1}{2\pi^2}\int_{x= -\infty}^{+\infty}\int_{y=0}^{+\infty}\int_{s\in \mathbb{S}(\Hq_0)}1_{x^2+y^2\leq 1}y^{-2}f(x+ys)4\pi y^2 d s d y d x\\
&=&\!\!\!\! \f{1}{2\pi^2}\int_{(x_1, \ldots, x_4)\in \R^4 }\!\!\!\! 1_{x_1^2+x_2^2+x_3^2+x_4^2\leq 1}(x_2^2+x_3^3+x_4^2)^{-1}\\
& & \qquad \qquad \qquad \qquad \times f(x_1+x_2\iq+x_3\jq+x_4\kq) d x_1 d x_2 d x_3 d x_4,
\end{eqnarray*}
which proves the theorem.
\end{proof}

 \vskip 1.5cm

\section{Proof of Theorem \ref{nonquadpotthm}}  

 \vskip 0.5cm

  Recall that $\restriction{\mathcal U}{\Sph(\Hq)}$ denotes the uniform measure on the sphere $\Sph(\Hq)$.   Theorem \ref{nonquadpotthm} will be proved as follows. We first identify the limit measure $\nu$ of the empirical distribution $\frac{1}{2n}\sum_{i=1}^{2n}\delta_{z_{n,i}}$ such that the measure 
  $\frac{1}{2n}\sum_{i=1}^{2n} \delta_{u_iz_{n,i}u_i^*}$ converges to  $\restriction{\mathcal U}{\Sph(\Hq)}$ for $(u_i)_{i\geq1}$ a sequence of i.i.d. $\restriction{\mathcal U}{\Sph(\Hq)}$-distributed random variables independent of the $z_{n,i}$'s. Using the characterization of the equilibrium measure given by Theorem \ref{caracpot}, we next show by contradiction that there is no quadratic potential $V$ such that $\nu$ is the equilibrium measure of $V$. It follows directly that there is no $V$ \st   for $(z_{n,1},\ldots,z_{n,n})$ distributed according to $P_n^V$, $\frac{1}{2n}\sum_{i=1}^{2n}\delta_{z_{n,i}}$ tends to $\nu$.\\

\noindent{\bf Claim: }for any sequence $(z_{n,1}, \ldots, z_{n,2n})\in \C^{2n}$ \st $z_{n,n+i}=\overline{z}_{n,i}$ and which is independent of the $u_j$'s, we have 
\[
\frac{1}{2n}\sum_{i=1}^{2n} \delta_{u_iz_{n,i}u_i^*}\underset{n\to\infty}{\longrightarrow} \restriction{\mathcal U}{\Sph(\Hq)} \text{ a.s.}\iff
\frac{1}{2n}\sum_{i=1}^{2n} \delta_{z_{n,i}}\underset{n\to\infty}{\longrightarrow} d\nu(z):=2|\Im(z)|^2 d\sigma(z) \text{ a.s.},
\]
where $\sigma$ is the Haar probability measure on the sphere $\Sph^1$.

Let us prove this claim. 
By Lemma \ref{empirique+tensor.25.11.08}, it suffices to prove that if $t$ is a complex random variable distributed according to $d\nu(z)=2|\Im(z)|^2 d\sigma(z)$ independent of a $\Sph(\Hq)$-uniformly distributed variable $u$,   $utu^*$ is uniformly distributed on $\Sph(\Hq)$.

By Proposition \ref{26.11.08.1}, for such random variables  $t$ and   $u$, $$utu^*\overset{(d)}{=}\Re(utu^*)+|\Im(utu^*)|S,$$for  $S$ uniformly distributed on the unit sphere of $\Hq_0=\{q\in \Hq \ste \Re(q)=0\}$ and   independent of $utu^*$. Thus since $|utu^*|=1$, we have $|\Im(utu^*)|=\sqrt{1-\Re(utu^*)^2}$. Moreover,  $\Re(utu^*)=\Re(t)$ (by  \cite[Th. 2.1.6]{z97} and because $t\in \C$),  so that  $$utu^*\overset{(d)}{=}\Re(t)+\sqrt{1-\Re(t)^2}S.$$Of course, the same applies to $u$:  one also has $$u \overset{(d)}{=}\Re(u )+\sqrt{1-\Re(u)^2}S,$$for  $S$ uniformly distributed on the unit sphere of $\Hq_0=\{q\in \Hq \ste \Re(q)=0\}$ and is independent of $u$.

So to prove the claim, it suffices to prove that $$\Re(t)  \overset{(d)}{=}\Re(u ),$$ which can easily be verified (the common distribution of $\Re(t)  $ and $\Re(u )$  is the semi-circular law on $[-1,1]$). The claim is proved.

The logarithmic potential of $\nu$ is given, for $x\in\C$ by
\[
U^{\nu}(x)=\int_{\Sph^1}\log|x-z|^{-1} d\nu(z)=
\begin{cases}
-\frac{1}{4}(\Re(x)^2-\Im(x)^2), & \text{if $|x|\leq 1$},\\
-\frac{1}{4|x|^4}(\Re(x)^2-\Im(x)^2)-\log |x|, & \text{if $|x|>1$}.
\end{cases}
\]
The detailed calculation of this fact is given in the appendix.

Let $z=x+iy\in\C$. Suppose by contradiction that there exists a quadratic potential $V(z)=a_1z^2+a_2\bar z^2+a_3z\bar z$, such that $\nu$ is the equilibrium measure of $V$. 
Since $V$ must be real valued, we have $a_1=a+ib=\bar a_2$, and $a_3=c\in\R$. So $V$ can be written $V(z)=x^2(a+c)+y^2(-a+c)+xy(-2b+c)$, for $z=x+iy$. 
By Theorem \ref{caracpot}, $2U^\nu+V$ must be constant on the support of $\nu$. So, we get, for $|z|=1$, $2U^\nu(z)+V(z)=(a+c-\frac{1}{2})x^2+(-a+c+\frac{1}{2})y^2+xy(-2b+c)=\text{constant}$. This implies that $a=\frac{1}{2}$, and $c=2b>0$. Hence,
\[
2U^\nu(z)+V(z)=c|z|^2=c\quad\text{on $|z|=1$}.
\]
But, on $|z|<1$, we have $2U^\nu(z)+V(z)<c$, which refutes the condition (ii) of Theorem \ref{caracpot}. So, we obtain a contradiction, and there does not exist a quadratic potential $V$ such that $\nu$ is the equilibrium measure of $V$.\hfill$\square$

 \vskip 1.5cm

\section{Appendix}

  \vskip 0.5cm
 
In this appendix, we give the calculation of the logarithmic potential of the measure $ d\nu(z)=2|\Im(z)|^2 d\sigma(z)$ used in the proof of Theorem \ref{nonquadpotthm} (so $d\sigma$ is the Haar \pro measure on the unit sphere $\Sp^1$). We begin with two lemmas.

\begin{lem}
We have, for $r\geq0$,
\[
\int_{-\pi}^{\pi} \log |r-e^{i\theta}| \frac{1}{\pi}\sin^2(\theta) d\theta=
\begin{cases}
\frac{r^2}{4},& \text{if $r<1$,} \\
\frac{1}{4r^2}+\log r,& \text{if $r>1$}.
\end{cases}
\] 
\end{lem}
\begin{proof}
Let $I=\int_{-\pi}^{\pi} \log |r-e^{i\theta}| \sin^2(\theta) d\theta=\frac{1}{2}\int_{-\pi}^{\pi}\log(r^2+1-2r\cos\theta)\sin^2(\theta) d\theta$. A first integration by parts gives
\[
I=-\frac{r^2+1}{4r}\!\!\!\underbrace{\int_{-\pi}^{\pi}\!\!\!\!\!\log(r^2+1-2r\cos\theta)\cos\theta d\theta}_{:=A}+\frac{1}{2}\!\!\underbrace{\int_{-\pi}^{\pi}\!\!\!\!\!\log(r^2+1-2r\cos\theta) d\theta}_{:=B}-I-\frac{\pi}{2}.
\]
By Lemma \ref{meanclassical}, we have
\[
B=2\int_{-\pi}^{\pi}\log|r-e^{i\theta}| d\theta=
\begin{cases}
0, & \text{if $r<1$,}\\
4\pi\log r, & \text{if $r>1$}.
\end{cases}
\]
Now let us calculate $A$. Integration by parts gives,
\[
A=-2r\underbrace{\int_{-\pi}^{\pi}\frac{1}{r^2+1-2r\cos\theta} d\theta}_{:=A_1}+2r\underbrace{\int_{-\pi}^{\pi}\frac{\cos^2\theta}{r^2+1-2r\cos\theta} d\theta}_{:=A_2}.
\]
Let $P_a(\theta)=\frac{1-a^2}{a^2+1-2a\cos\theta}$ be the Poisson kernel, for $0\leq a<1$. Recall   that
\[
\frac{1}{2\pi}\int_{-\pi}^{\pi}P_a(t-\theta)e^{in\theta} d\theta=a^{|n|}e^{int}, \qquad \text{ for } n\in\Z.
\]
Thus, we obtain
\[
A_1=
\begin{cases}
\frac{2\pi}{1-r^2}, & \text{for $r<1$,}\\
\frac{2\pi}{r^2-1}, & \text{for $r>1$},
\end{cases}
\]
the second case being obtained by replacing $r$ by $\frac{1}{r}$.
In the same way,  
\[
A_2=
\begin{cases}
\frac{1}{1-r^2}\pi(r^2+1), & \text{for $r<1$},\\
\frac{1}{r^2-1}\pi(\frac{1}{r^2}+1), & \text{for $r>1$},
\end{cases}
\] which allows to conclude.\end{proof}

To determine the logarithmic potential for points belonging to the unit circle, we will need the following lemma, whose proof is straightforward using change of variables and mean value property of harmonic functions.

\begin{lem}
We have
\[
\int_{-\pi}^{\pi}\log|1-e^{i\theta}|\frac{1}{\pi}\sin^2\theta d\theta=\frac{1}{4}.
\]
\end{lem}

By the two previous lemmas,  
\begin{equation} \label{Vpot-sin}
\int_{-\pi}^{\pi}\log|r-e^{i\theta}|\frac{1}{\pi}\sin^2\theta d\theta=
\begin{cases}
\frac{r^2}{4}, & \text{if $0\leq r\leq1$},\\
\frac{1}{4r^2}+\log r, & \text{if $r>1$}.
\end{cases}
\end{equation}

The logarithmic potential of the measure $ d\nu(z)=2|\Im z|^2 d\sigma(z)$ on $\Sph^1$ can now be determined.
\begin{prop}
Let $ d\nu(z)=2|\Im z|^2 d\sigma(z)$. The logarithmic potential of $\nu$ is given, for $x\in\C$,  by
\[
U^{\nu}(x)=\int_{\Sph^1}\log|x-z|^{-1} d\nu(z)=
\begin{cases}
-\frac{1}{4}(\Re(x)^2-\Im(x)^2), & \text{if $|x|\leq 1$},\\
-\frac{1}{4|x|^4}(\Re(x)^2-\Im(x)^2)-\log |x|^2, & \text{if $|x|>1$}.
\end{cases}
\]
\end{prop}
\begin{proof}
Let $x=re^{i\gamma}\in\C$. 
\[
\int_{\Sph^1} \log|re^{i\gamma}-z| d\nu(z)=\int\log|r-z|\frac{1}{\pi}|\Im(ze^{i\gamma})|^2 d\sigma(z),
\]
since $ d\sigma$ is translation-invariant. It follows that
\beq 
&&\int_{\Sph^1} \log|re^{i\gamma}-z| d\nu(z)=\\&&
\int_{-\pi}^{\pi}\log|r-e^{i\theta}|\frac{1}{\pi}\sin^2\theta d\theta \cos^2\gamma+\int_{-\pi}^{\pi}\log|r-e^{i\theta}|\frac{1}{\pi}\cos^2\theta d\theta\sin^2\gamma\\
&& \qquad\qquad\qquad\qquad+\int_{-\pi}^{\pi}\log|r-e^{i\theta}|\frac{1}{\pi}\sin\theta\cos\theta d\theta 2\sin\gamma\cos\gamma.
\eeq
The last integral is zero by parity, so
the result follows by Lemma \ref{meanclassical} and the previous calculation (\ref{Vpot-sin}).
\end{proof}



\providecommand{\bysame}{\leavevmode\hbox to3em{\hrulefill}\thinspace}
\providecommand{\MR}{\relax\ifhmode\unskip\space\fi MR }
\providecommand{\MRhref}[2]{%
  \href{http://www.ams.org/mathscinet-getitem?mr=#1}{#2}
}
\providecommand{\href}[2]{#2}

\end{document}